\documentclass[final,5p,times,twocolumn]{elsarticle}
\usepackage[T1]{fontenc}
\usepackage{microtype}
\usepackage{amsmath,amssymb,amsthm,mathtools,enumitem,booktabs}
\usepackage{graphicx,tikz}
\usepackage{float}
\floatstyle{plain}
\newfloat{algorithm}{tbp}{loa}
\floatname{algorithm}{Algorithm}
\usepackage{pgfplots}
\pgfplotsset{compat=1.18}
\usetikzlibrary{arrows.meta,positioning,shapes.geometric}
\definecolor{cellblue}{RGB}{40,95,150}
\definecolor{cellorange}{RGB}{190,100,25}
\usepackage[hidelinks]{hyperref}
\newtheorem{theorem}{Theorem}[section]
\newtheorem{lemma}[theorem]{Lemma}
\newtheorem{proposition}[theorem]{Proposition}
\newtheorem{corollary}[theorem]{Corollary}
\theoremstyle{definition}\newtheorem{definition}{Definition}[section]
\theoremstyle{remark}\newtheorem{remark}[theorem]{Remark}
\newcommand{\R}{\mathbb R}\newcommand{\Q}{\mathbb Q}\newcommand{\Z}{\mathbb Z}
\newcommand{\eps}{\varepsilon}
\newcommand{\poly}{\operatorname{poly}}

\makeatletter
\let\ps@pprintTitle\ps@plain
\makeatother
\begin{document}
\begin{frontmatter}
\title{A short proof of a range-relative FPTAS for polynomial integer optimization over bounded polyhedra}
\author{Robert Hildebrand\corref{cor1}}
\cortext[cor1]{Corresponding author.}
\ead{rhil@vt.edu}
\affiliation{organization={Grado Department of Industrial and Systems Engineering, Virginia Tech},
  city={Blacksburg}, state={Virginia}, postcode={24061}, country={USA}}
\begin{abstract}
We give a short proof of a fully polynomial-time range-relative approximation scheme for polynomial optimization over the integer points of a bounded rational polyhedron in fixed dimension. The proof combines the reflecting-set decomposition of Cook, Hartmann, Kannan, and McDiarmid with Markov's polynomial derivative inequality and uses one integer linear feasibility query per refined cell. The running time is polynomial in the input encoding size, numerical degree, and reciprocal accuracy.
\end{abstract}
\end{frontmatter}

\section{Introduction and main result}

Let $P=\{\mathbf{x}\in\R^d:A\mathbf{x}\le\mathbf b\}$ be a bounded rational polyhedron, with $A\in\Q^{m\times d}$ and $\mathbf b\in\Q^m$, and let $f\in\Z[x_1,\ldots,x_d]$. We study
\begin{equation}\label{eq:problem}
 \begin{aligned}
 \min\quad & f(\mathbf{x})\\
 \text{s.t.}\quad & A\mathbf{x}\le\mathbf b,\\
 & \mathbf{x}\in\Z^d.
 \end{aligned}
\end{equation}
When $P\cap\Z^d\ne\varnothing$, choose minimizers and maximizers $\mathbf{x}_{\min},\mathbf{x}_{\max}$ and write $f_{\min}=f(\mathbf{x}_{\min})$ and $f_{\max}=f(\mathbf{x}_{\max})$.

De Loera, Hemmecke, K\"oppe, and Weismantel established an FPTAS for nonnegative polynomial maximization and a range-relative scheme for arbitrary polynomials in fixed dimension~\cite{DHKW06,DHKW08}. We give a short geometric proof for the bounded integer problem, using integer linear feasibility and an explicit bit-complexity bound. The guarantee is invariant under translation of the objective; it is not a multiplicative approximation to a nonnegative minimum.

An interactive explanation, supporting material, and partial Lean verification are available in the project repository.\footnote{\url{https://github.com/RobertHildebrand/FPTAS-Polynomial-Optimization}}

For continuous convex bodies, Piazzon and Vianello~\cite[Propositions~1--2]{PV} obtain range-relative approximation from polynomial norming meshes, selecting an arbitrary feasible representative in each sufficiently small box and applying Markov's inequality. Our refined reflecting cells provide feasible integer progressions, allowing error to be measured against the integer objective range rather than the continuous range.

The key ingredient is decomposition by relative constraint slacks. The reflecting sets of Cook, Hartmann, Kannan, and McDiarmid~\cite{CHKM} underpin their integer-hull analysis; related symmetric covers appear in the quadratic and cubic algorithms of Ari and Hildebrand~\cite{AH}, and dyadic slack grouping also appears in Wei's quadratic algorithm~\cite[Section~4.1]{Wei}. Here refinement produces long feasible integer progressions. Markov's polynomial derivative inequality~\cite[Theorem~5.1.8]{BorweinErdelyi} on these progressions controls the error at an integer optimum, allowing any feasible representative of its cell. The same representatives work for all polynomials of a prescribed degree.

Markov's inequality is essential to the polynomial cell bound: it controls variation relative to the integer objective range, so the required refinement depends only on degree and accuracy. A coefficient-based Lipschitz estimate alone could require refinement exponential in the input encoding size. The Lipschitz-based grid reduction of~\cite{DHKW08} can tolerate a large denominator because its integer solver depends polynomially on the denominator's bit length; explicit cell enumeration requires polynomially many cells.

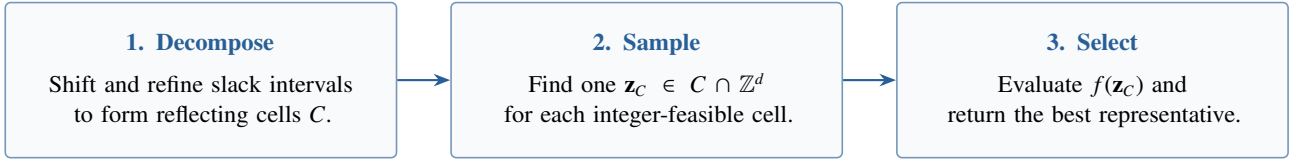
\begin{figure*}[t]
\centering
\begin{tikzpicture}[
  stage/.style={draw=cellblue!65,line width=.6pt,rounded corners=2pt,
    fill=cellblue!3,text width=.255\textwidth,minimum height=2.05cm,
    inner sep=7pt,align=center,font=\small},
  flow/.style={-{Stealth[length=2.2mm]},line width=.8pt,cellblue},
  node distance=7mm]
\node[stage] (decompose) {
  {\color{cellblue}\bfseries 1. Decompose}\\[5pt]
  Shift and refine slack intervals\\
  to form reflecting cells $C$.
};
\node[stage,right=of decompose] (sample) {
  {\color{cellblue}\bfseries 2. Sample}\\[5pt]
  Find one $\mathbf z_C\in C\cap\Z^d$\\
  for each integer-feasible cell.
};
\node[stage,right=of sample] (select) {
  {\color{cellblue}\bfseries 3. Select}\\[5pt]
  Evaluate $f(\mathbf z_C)$ and\\
  return the best representative.
};
\draw[flow] (decompose)--(sample);
\draw[flow] (sample)--(select);
\end{tikzpicture}
\caption{Algorithm at a glance: decompose a relaxation with the same integer points, select one representative from each cell containing integer points, and return the representative with the best objective value.}\label{fig:overview}
\end{figure*}

The algorithm follows the decompose--sample--select procedure summarized in Figure~\ref{fig:overview}. No local objective model is optimized. The proof separates the reflecting geometry (Section~\ref{sec:cover}) from enumeration and bit complexity (Section~\ref{sec:algorithm}).

To state the complexity bound precisely, we fix the following encoding convention. Vectors are columns and are written in bold. For an integer $u$, its \emph{binary encoding size} is $\langle u\rangle=1+\lceil\log_2(1+|u|)\rceil$, counting a sign bit and the binary magnitude bits. We use \emph{encoding size}, \emph{bit length}, or simply \emph{size} for storage measured in bits when the context is clear; \emph{magnitude} refers to the numerical absolute value. A rational number $u/v$ is stored in lowest terms with $v>0$, with size $\langle u/v\rangle=\langle u\rangle+\langle v\rangle$. For matrices, vectors, and lists, $\langle\cdot\rangle$ is the sum of the entry sizes together with the binary encoding sizes of the dimensions and list lengths. The polynomial is supplied explicitly as
\[
 f(\mathbf{x})=\sum_{\boldsymbol{\alpha}\in\mathcal A}c_{\boldsymbol{\alpha}}\mathbf{x}^{\boldsymbol{\alpha}},
 \qquad c_{\boldsymbol{\alpha}}\in\Z,\quad \boldsymbol{\alpha}\in\Z_+^d,\quad |\boldsymbol{\alpha}|_1\le D.
\]
Following~\cite{DHKW08}, coefficients are encoded in binary and exponent vectors in unary: each exponent $\alpha_j$ is stored as $\alpha_j$ one-bits followed by a zero, using $\alpha_j+1$ bits. The input is therefore a bit string, although exponents use unary rather than positional binary notation. Accordingly, $\langle f\rangle$ includes the coefficient sizes and $\sum_{\boldsymbol{\alpha}\in\mathcal A}\sum_{j=1}^d(\alpha_j+1)$, together with dimensions and list lengths. Write $\varphi=\langle A,\mathbf b,f,\eps\rangle$ for the total input encoding size. We set $D=\max\{1,\deg f\}$ from the nonzero monomials, taking $D=1$ for the zero polynomial; no separate degree input is needed. We retain $D$ explicitly in the running-time bound to show its contribution.

\begin{definition}[Range-relative FPTAS]\label{def:fptas}
Under the encoding convention above and for fixed dimension $d$, a range-relative fully polynomial-time approximation scheme is an algorithm that, for every rational $0<\eps\le1$, reports infeasibility when $P\cap\Z^d=\varnothing$, and otherwise returns $\mathbf{x}_{\eps}\in P\cap\Z^d$ satisfying
\begin{equation}\label{eq:guarantee}
 0\le f(\mathbf{x}_{\eps})-f_{\min}\le\eps(f_{\max}-f_{\min}).
\end{equation}
Its running time is polynomial in the total input encoding size $\varphi$ and $1/\eps$. The polynomial may depend on the fixed dimension. This is the fully polynomial-time weak-approximation notion of~\cite{DHKW08}. The maximization definition uses $f_{\max}-f(\mathbf{x}_{\eps})$ on the left.
\end{definition}

Running times are measured in bit operations; oracle calls and arithmetic operations are counted separately where stated. The function $\poly$ denotes a polynomial with an absolute exponent, independent of $d,D,\eps$; different occurrences may denote different polynomials. Logs are base two unless indicated otherwise.

\begin{theorem}[Polynomial approximation by reflecting cells]\label{thm:main}
Let $d\ge1$ and let $P=\{\mathbf{x}\in\R^d:A\mathbf{x}\le\mathbf b\}$ be bounded, with $A\in\Z^{m\times d}$ and $\mathbf b\in\Z^m$ (rational rows are normalized in Section~2), let $f\in\Z[x_1,\ldots,x_d]$, and put $D=\max\{1,\deg f\}$, with $D=1$ when $f\equiv0$. For rational $0<\eps\le1$, let $\varphi=\langle A,\mathbf b,f,\eps\rangle$ be the total input encoding size.
There is a deterministic algorithm that reports infeasibility or returns $\mathbf{x}_{\eps}\in P\cap\Z^d$ satisfying~\eqref{eq:guarantee}, in at most
\begin{equation}\label{eq:runtime-coarse}
 \left(\frac{\varphi D}{\eps}\right)^{O(d)}
\end{equation}
bit operations, with an absolute constant in the $O(d)$ exponent.
\end{theorem}

Since the input explicitly lists the $m$ inequalities in $d$ variables, $m,d\le\varphi$; their separate factors in the detailed running-time bound are therefore absorbed into the base $\varphi$ and the $O(d)$ exponent in~\eqref{eq:runtime-coarse}. The dependence on dimension remains explicit in that exponent. More precisely, for fixed $P,D,d$, the algorithm makes $O(\eps^{-d/2})$ integer-feasibility queries; the detailed running-time bound in the proof retains the factor $\max\{D^2,D/\sqrt\eps\}$.

Rational inequalities can be normalized to integral data in polynomial time, as detailed below. Rational objective coefficients can likewise be reduced to integer coefficients by clearing denominators in polynomial time; positive scaling preserves the approximation guarantee, as in~\cite[Section~5]{DHKW08}. The remaining construction parameters and an explicit upper bound on the number of feasibility queries are introduced in the proof. For fixed $d$, the bound is polynomial in $\varphi$ and $1/\eps$, hence gives an FPTAS in Definition~\ref{def:fptas}. The same explicit bound remains valid for binary-encoded exponents, but then its numerical dependence on $D$ does not assert an FPTAS in the sparse binary input size.

The following nonnegative-maximization guarantee was established by De Loera et al.~\cite[Theorem~1.1]{DHKW06}; it follows from our range-relative bound by the same elementary conversion.

\begin{corollary}[Nonnegative maximization {\cite[Theorem~1.1]{DHKW06}}]\label{cor:max}
Under the same assumptions and encoding conventions, if $f\ge0$ on $P\cap\Z^d$, choosing the largest objective value among the representatives yields
\[
 f(\mathbf{x}_{\eps})\ge(1-\eps)f_{\max}.
\]
Thus the construction gives a multiplicative FPTAS for nonnegative maximization under the encoding convention of Definition~\ref{def:fptas}.
\end{corollary}

For the mixed-integer result, we follow the integral-slice scaling, barycentric rounding, and reduction to integer optimization of De Loera et al.~\cite[Section~3]{DHKW08}, who established the approximation guarantees in~\cite[Theorems~1(a) and~2(b)]{DHKW08}. We improve the grid multiplier in this reduction by using Markov's inequality in place of their coefficient-based Lipschitz estimate. Combining the resulting coefficient-independent multiplier with Theorem~\ref{thm:main} yields the explicit running-time bound below.

\begin{corollary}[Mixed-integer approximation]\label{cor:mixed}
Fix the total dimension $d=d_c+d_i$. Let $P=\{\mathbf{x}\in\R^d:A\mathbf{x}\le\mathbf b\}$ be a bounded rational polyhedron and $f\in\Z[x_1,\ldots,x_d]$. Under the same encoding convention, for rational $0<\eps\le1$ there is an algorithm, polynomial in $\varphi=\langle A,\mathbf b,f,\eps\rangle$ and $1/\eps$, that reports infeasibility or returns
\[
\begin{gathered}\mathbf{x}_{\eps}\in P\cap(\Q^{d_c}\times\Z^{d_i}),\\
 0\le f(\mathbf{x}_{\eps})-f_{\min}\le\eps(f_{\max}-f_{\min}),
\end{gathered}\]
where the extrema are over $P\cap(\R^{d_c}\times\Z^{d_i})$. For nonnegative $f$ on this set, a multiplicative FPTAS for maximization is obtained as well. The running time is at most
\begin{equation}\label{eq:mixed-runtime}
 \left(\frac{\varphi D}{\eps}\right)^{O(d(d_c+1))}
\end{equation}
bit operations, with an absolute constant in the exponent.
\end{corollary}

\paragraph{Comparison with prior algorithms}

The short rational generating functions used by De Loera et al. build on Barvinok's lattice-point counting algorithm~\cite{Barvinok}. Their approach~\cite{DHKW06,DHKW08} computes weighted sums involving $f^q$, where $q=O(1+\eps^{-1}\log N)$ and $N=|P\cap\Z^d|$. The standard Barvinok counting bound is $\langle A,\mathbf b\rangle^{O(d\log(d+1))}$~\cite[Theorem~2.5]{PakPanova} (under their arithmetic-cost convention; see Section~8.8); our entire algorithm has bound $\varphi^{O(d)}$ for fixed $D,\eps$. This compares available upper bounds, not a runtime separation: counting alone does not describe their FPTAS, and our potentially large cell count does not establish practical superiority.

Del Pia's quadratic scheme~\cite[Proposition~3 and Section~6.2]{DelPia} uses affine optimization over boxes after geometric preprocessing, with flatness recursion in the remaining cases. For fixed quadratic rank $k$ and integer dimension $p$, it uses $O_{k,p}(1+\eps^{-k/2})$ MILP subproblems and allows variable continuous dimension. Our feasibility-call exponent matches $k/2$ at $k=d$, but our enumeration costs an additional factor of order $\eps^{-1/2}$. Our advantages are arbitrary degree and objective-independent representatives; we do not recover his low-rank dependence.

\section{Reflecting cells and approximation}\label{sec:cover}

We first construct cells supporting feasible integer progressions, then apply Markov's inequality along a progression through an integer optimum. Neither the optimum nor the objective range is needed by the algorithm.

Clear rational row denominators by multiplying each row by the product of its positive denominators. This takes polynomial time; normalized entries have $O(\varphi)$ bits. Equalities may be represented by pairs of inequalities, and lower-dimensional polytopes are allowed. Real feasibility and boundedness are testable by rational linear programming; boundedness remains a hypothesis of our theorem. After testing real feasibility, remove zero constraint rows. We treat $d\ge1$ below; dimension zero is immediate.

For the remainder of the construction, $A,\mathbf b$ denote the normalized integral data; $\varphi$ continues to denote the original input size. To track the construction explicitly, define
\begin{equation}\label{eq:tau}
 \tau=\left\lceil\log_2\bigl(2+\max\{\|A\|_{\max},\|\mathbf b\|_\infty\}\bigr)\right\rceil.
\end{equation}
In Schrijver's encoding convention~\cite[Section~2.1]{Schrijver}, each normalized row inequality has size at most $(d+1)(\tau+3)$, including its right-hand side. This bounds the per-inequality size parameter used by CHKM, distinct from our total input size $\varphi$. Here $\|A\|_{\max}$ is the largest absolute matrix entry, so $\tau$ bounds the coefficient bit lengths up to a constant. The remaining parameters have the following roles: $\Phi$ gives the common magnitude bound $2^\Phi$ for coordinates and shifted slacks and determines how many intervals $[2^k,2^{k+1}]$ are needed below, $\ell$ is the refinement, and $H$ bounds the number of hyperplanes. Set
\begin{align}
 \Phi&=5d^2(d+1)(\tau+3)+2,\label{eq:Phi}\\
 \ell&=\max\{D^2,\lceil D/\sqrt\eps\rceil\},\qquad H=m((\Phi+2)\ell+1),\label{eq:HT}\\
 R_d(H)&=\sum_{j=0}^{d}\binom Hj\le(1+H)^d.\label{eq:F}
\end{align}

Following the outward perturbation in~\cite[proof of Theorem~2.1, p.~29]{CHKM}, define
\[
 \widetilde P:=\{\mathbf{x}:A\mathbf{x}\le\mathbf b+\tfrac12\mathbf1\},\quad
 s_i(\mathbf{x}):=b_i+\tfrac12-\mathbf a_i^\top\mathbf{x}.
\]
Since $A,\mathbf b$ are integral, $\widetilde P\cap\Z^d=P\cap\Z^d$. Every integer point of $P$ has $s_i\ge1/2$. We will extrapolate real segments in $\widetilde P$ and use integrality to return their integer samples to $P$. This shift is applied only to the pure-integer problem, including the scaled integer problem in the mixed-integer reduction.

The coordinate and slack estimates in the proof of CHKM's Theorem~2.1~\cite[p.~29]{CHKM}, with their per-inequality size bounded by $(d+1)(\tau+3)$, give
\begin{equation}\label{eq:integer-bounds}
 |x_j|<2^\Phi,\qquad \tfrac12\le s_i(\mathbf{x})<2^\Phi
 \quad(\mathbf{x}\in P\cap\Z^d).
\end{equation}
Their estimates are stated at vertices of the integer hull and extend throughout that hull by linearity. Their integral shifted rows $2A\mathbf{x}\le2\mathbf b+\mathbf1$ have slacks $2s_i$; the extra constant in~\eqref{eq:Phi} safely covers this convention. Their argument treats $d\ge2$; for $d=1$, the same bounds follow directly from the interval endpoints and $|a_{ij}|,|b_i|<2^\tau$. Only bounds on integer feasible points are needed here. The extended version retains the detailed determinant proof with a sharper exponent.

We use the reflecting-set construction of CHKM~\cite[proof of Theorem~2.1, pp.~29--30]{CHKM}: move integer points into the interior and group slacks in dyadic intervals. Our modification subdivides each such interval into $\ell$ equal parts. This extends their reflection calculation to long integer progressions; see also~\cite[Theorem~2.2]{AH}. Figure~\ref{fig:pentagon-cells} illustrates the refinement.

\begin{definition}[Closed refined cells]\label{def:thresholds}
For integers $\Phi,\ell\ge1$, let
\begin{equation}\label{eq:thresholds}
\begin{split}
\mathcal B_{\Phi,\ell}=\{2^\Phi\}\cup\\
 \left\{2^k\left(1+\frac j\ell\right):-2\le k<\Phi,\ 0\le j<\ell\right\}.
\end{split}
\end{equation}
There are $(\Phi+2)\ell+1$ thresholds, starting at $1/4$. Set
\[
 Q=\{\mathbf{x}:1/4\le s_i(\mathbf{x})\le2^\Phi\ (1\le i\le m)\}.
\]
Subdivide $Q$ by the hyperplanes $s_i(\mathbf{x})=a$, $a\in\mathcal B_{\Phi,\ell}$. The family $\mathcal C$ consists of the closures of the resulting full-dimensional chambers. If $Q$ has empty interior, set $\mathcal C=\varnothing$.
\end{definition}

Every integer point of $P$ lies in the interior of $Q$, by~\eqref{eq:integer-bounds}, and therefore belongs to a closed cell. The cells may overlap on subdivision hyperplanes, but no separate lower-dimensional faces are needed. Conversely, $Q\subseteq\widetilde P$, so every integer point in a cell belongs to $P$. These statements also justify declaring integer infeasibility when $Q$ has empty interior. The set $Q$ is bounded: its recession directions satisfy $A\mathbf v=0$, which forces $\mathbf v=0$ since $P$ is nonempty and bounded.

If $a<b$ are successive thresholds, then
\begin{equation}\label{eq:width}
 b-a\le a/\ell.
\end{equation}
Indeed, on $[2^k,2^{k+1}]$ the spacing is $2^k/\ell$ and $a\ge2^k$. The same argument applies to $[1/4,1/2]$ and $[1/2,1]$.

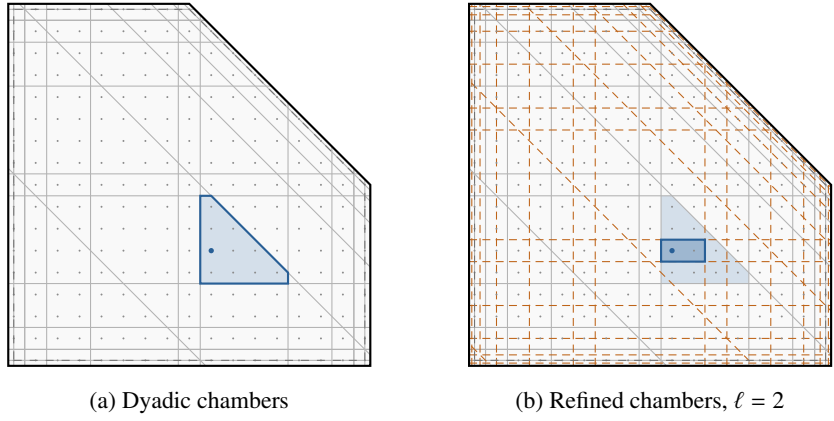
\begin{figure*}[t]
\centering
\begin{tikzpicture}[x=.29cm,y=.29cm,font=\small]
\foreach \shift/\panel/\refined in {0/{(a) Dyadic chambers}/0,21/{(b) Refined chambers, $\ell=2$}/1} {
 \begin{scope}[xshift=\shift*.29cm]
  \begin{scope}
   \clip (-.25,-.25)--(16.25,-.25)--(16.25,8)--(8,16.25)--(-.25,16.25)--cycle;
   \fill[gray!5] (-.25,-.25) rectangle (16.25,16.25);
   \fill[cellblue!20] (8.5,3.5)--(12.5,3.5)--(12.5,4)--(9,7.5)--(8.5,7.5)--cycle;
   \foreach \a in {.25,.5,1,2,4,8,16,32} {
    \draw[gray!60,thin] ({\a-.5},-.25)--({\a-.5},16.25);
    \draw[gray!60,thin] ({16.5-\a},-.25)--({16.5-\a},16.25);
    \draw[gray!60,thin] (-.25,{\a-.5})--(16.25,{\a-.5});
    \draw[gray!60,thin] (-.25,{16.5-\a})--(16.25,{16.5-\a});
    \draw[gray!60,thin] (-10,{34.5-\a})--(35,{-10.5-\a});
   }
   \ifnum\refined=1
    \foreach \a in {.375,.75,1.5,3,6,12,24} {
     \draw[cellorange,densely dashed,thin] ({\a-.5},-.25)--({\a-.5},16.25);
     \draw[cellorange,densely dashed,thin] ({16.5-\a},-.25)--({16.5-\a},16.25);
     \draw[cellorange,densely dashed,thin] (-.25,{\a-.5})--(16.25,{\a-.5});
     \draw[cellorange,densely dashed,thin] (-.25,{16.5-\a})--(16.25,{16.5-\a});
     \draw[cellorange,densely dashed,thin] (-10,{34.5-\a})--(35,{-10.5-\a});
    }
    \fill[cellblue!45] (8.5,4.5) rectangle (10.5,5.5);
    \draw[cellblue,thick] (8.5,4.5) rectangle (10.5,5.5);
   \else
    \draw[cellblue,thick] (8.5,3.5)--(12.5,3.5)--(12.5,4)--(9,7.5)--(8.5,7.5)--cycle;
   \fi
   \foreach \x in {0,...,16} {\foreach \y in {0,...,16} {
    \pgfmathtruncatemacro{\inside}{\x+\y<=24}
    \ifnum\inside=1 \fill[black!45] (\x,\y) circle (.045);\fi
   }}
  \end{scope}
  \draw[black!55,dashed] (0,0)--(16,0)--(16,8)--(8,16)--(0,16)--cycle;
  \draw[thick] (-.25,-.25)--(16.25,-.25)--(16.25,8)--(8,16.25)--(-.25,16.25)--cycle;
  \fill[cellblue] (9,5) circle (.12);
  \node[below] at (8,-1) {\panel};
 \end{scope}
}
\end{tikzpicture}
\caption{Closed shifted-slack chambers for $P=\{\mathbf{x}\in[0,16]^2:x_1+x_2\le24\}$ (dashed outline). The solid outline is $Q$, whose lower shifted-slack bound is $1/4$; the upper bound is inactive in this example. The fifth shifted slack is $24.5-x_1-x_2$, giving the diagonal lines. The shaded dyadic chamber is refined in (b); its darker closed rectangle $[8.5,10.5]\times[4.5,5.5]$ contains the representative $(9,5)$. All lattice points of $P$, including its boundary points, lie inside $Q$. Adjacent closed cells share their boundaries.}
\label{fig:pentagon-cells}
\end{figure*}

\begin{lemma}[Refined CHKM reflection property]\label{lem:reflection}
Each $C\in\mathcal C$ is described by $a_i\le s_i(\mathbf{x})\le c_i$ for successive thresholds $a_i<c_i$. Set
\[
 D_C=\{\mathbf v\in\R^d:|\mathbf a_i^\top\mathbf v|\le a_i\ (1\le i\le m)\}.
\]
Then $\ell(C-C)\subseteq D_C$ and $C\pm D_C\subseteq\widetilde P$. In particular,
\begin{equation}\label{eq:long}
 \mathbf{x}+k(\mathbf{y}-\mathbf{x})\in P\cap\Z^d
 \quad(\mathbf{x},\mathbf y\in C\cap\Z^d,\ k\in\Z,\ |k|\le\ell).
\end{equation}
\end{lemma}
\begin{proof}
The CHKM reflection calculation~\cite[p.~30]{CHKM} applies with the refined width:
\[
 |s_i(\mathbf y)-s_i(\mathbf{x})|\le c_i-a_i\le a_i/\ell
 \le s_i(\mathbf{x})/\ell.
\]
This gives both inclusions and the real progression in $\widetilde P$. Its integer samples belong to $P$ because the shift preserves integer feasibility. Boundedness of $D_C$ follows from $A\mathbf v=0\Rightarrow\mathbf v=0$.
\end{proof}

\begin{remark}[Relative slacks characterize extrapolation]\label{rem:slack-characterization}
Let $\mathbf{x},\mathbf y\in\widetilde P$, set $\mathbf v=\mathbf y-\mathbf{x}$, and let $\ell>0$. The segments $\mathbf{x}+[-\ell,\ell]\mathbf v$ and $\mathbf y+[-\ell,\ell]\mathbf v$, centered at $\mathbf{x}$ and $\mathbf y$ respectively, both lie in $\widetilde P$ if and only if, for every row $i$,
\[
 \ell\,|s_i(\mathbf y)-s_i(\mathbf{x})|
 \le\min\{s_i(\mathbf{x}),s_i(\mathbf y)\}.
\]
Indeed, the minimum slack along the first segment is $s_i(\mathbf{x})-\ell|s_i(\mathbf y)-s_i(\mathbf{x})|$, and similarly for the second. For positive slacks the condition is equivalently
\[
 \frac{\max\{s_i(\mathbf{x}),s_i(\mathbf y)\}}
 {\min\{s_i(\mathbf{x}),s_i(\mathbf y)\}}\le1+\frac1\ell;
\]
if either slack is zero, both must be zero. Our cells have positive shifted slacks, so this last case never arises. Thus multiplicative slack control is exactly the geometry needed for extrapolation about both points. Refined dyadic intervals enforce this condition with rational inequalities of polynomial encoding size, and the integer slack bounds limit the number of dyadic intervals needed. The condition is independent of the objective: Proposition~\ref{prop:extrema} turns the resulting cover into a simultaneous approximation guarantee for every polynomial of degree at most $D$.
\end{remark}

The integer progression property is the geometric input to approximation. Closed chambers are defined by non-strict linear inequalities; rounding the slack bounds below removes threshold denominators without changing the integer points.

\begin{lemma}[Number, encoding, and enumeration of cells]\label{lem:arrangement}
The closed cells can be enumerated deterministically in $(\varphi\ell)^{O(d)}$ bit operations, and their number is at most $(\varphi\ell)^{O(d)}$. Each cell has a description by at most $2m$ non-strict linear inequalities of total encoding size $\poly(\varphi,\log(\ell+1))$, hence $\poly(\varphi)$ for the prescribed refinement~\eqref{eq:HT}. Rounding its slack bounds as specified below gives a rational polytope $\widehat C\subseteq C$ with exactly the same integer points, described by at most $2m$ inequalities whose coefficient bit lengths are $O(\tau+\Phi+1)$, independent of $\ell$.
\end{lemma}
\begin{proof}
With $H=m((\Phi+2)\ell+1)$, the standard arrangement bound gives $R_d(H)$ chambers~\cite{Stanley}; intersecting with convex $Q$ cannot increase their number. The thresholds have denominator dividing $4\ell$ and numerator bit length $O(\Phi+\log(\ell+1))$.

Incrementally split each closed cell into its two closed halfspace intersections, retaining only those with nonempty interior. These cover the parent, including the splitting hyperplane. At most $R_d(H)$ cells occur at each stage. Their constraints compress to two bounds per original row. Interior is tested by maximizing $\delta\in[0,1]$ with a common margin $\delta$ in every inequality; positivity is equivalent to nonempty interior since zero normals were removed. Thus $O((H+1)R_d(H))$ rational LP tests suffice. Since $m,d,\tau=O(\varphi)$ and $\Phi=O(d^3(\tau+1))$, each test has encoding size $\poly(\varphi,\log(\ell+1))$ and the total enumeration cost is $(\varphi\ell)^{O(d)}$. For~\eqref{eq:HT}, $\log(\ell+1)=O(\log(D+1)+\langle\eps\rangle)$, giving the stated $\poly(\varphi)$ size under unary exponent encoding.

Finally, replace $a\le s_i(\mathbf{x})\le c$ by
\[
 \lceil a-\tfrac12\rceil\le b_i-\mathbf a_i^\top\mathbf{x}
 \le\lfloor c-\tfrac12\rfloor.
\]
These inequalities define $\widehat C\subseteq C$ with $\widehat C\cap\Z^d=C\cap\Z^d$, and remove the dependence on $\ell$ from the coefficient bit lengths. The extended version gives the enumeration details.
\end{proof}

\begin{remark}[Parallel families]
If row $i$ contributes $q_i$ distinct threshold hyperplanes, the count can be sharpened to
\[
 |\mathcal C|\le
 \sum_{\substack{I\subseteq\{1,\ldots,m\}\\|I|\le d}}
 \prod_{i\in I}q_i
 \le\sum_{j=0}^d\binom mj((\Phi+2)\ell+1)^j.
\]
The insertion argument counts intersections involving at most one hyperplane from each parallel family; further dependencies can only reduce the count. This improves the bound without changing its worst-case $\ell^d$ dependence. We retain $R_d(H)$ for simplicity. The boundary-vertex count in~\cite[pp.~30--31]{CHKM} bounds integer-hull vertices, not all reflecting cells, and therefore does not directly replace this count.
\end{remark}

We next combine the sample-to-supremum argument used in polynomial norming meshes~\cite[proof of Proposition~2]{PV} with a second-difference bound at a discrete minimum. Long progressions control variation between integer samples without assuming continuous stationarity.

We use the classical Markov inequality~\cite[Theorem~5.1.8]{BorweinErdelyi}: for a real polynomial $q$ of degree at most $D$ and an interval $[u,v]$ with $u<v$,
\[
 \max_{u\le t\le v}|q'(t)|
 \le\frac{2D^2}{v-u}\max_{u\le t\le v}|q(t)|.
\]
Thus the largest absolute derivative is bounded by the largest absolute polynomial value, multiplied by a factor depending only on the degree and interval length.

Bounds between equally spaced samples go back to Ehlich and Zeller~\cite{EhlichZeller}; Coppersmith and Rivlin~\cite{CoppersmithRivlin} quantify their dependence on degree and sample count, with sharp fixed-ratio asymptotics given by Kuijlaars~\cite{Kuijlaars}. The reduction of symmetric-node problems to polynomials in $t^2$ also appears in Rivlin~\cite[p.~539]{RivlinSurvey}, for coefficient bounds. We retain the elementary proof below for the discrete-minimum estimate needed here; these references concern its classical ingredients.

\begin{lemma}[Approximation at a discrete minimum]\label{lem:markov}
Let $D\ge1$ and $\ell\ge D^2$ be integers, and let $p\in\R[t]$ have degree at most $D$. Suppose
\[
 a\le p(k)\le b\quad(k=-\ell,\ldots,\ell),\qquad p(-1)\ge p(0).
\]
Then
\[
 p(1)-p(0)\le \frac{D^2}{\ell^2}(b-a).
\]
\end{lemma}
\begin{proof}
Set $q=p-(a+b)/2$ and $W=\max_{-\ell\le t\le\ell}|q(t)|$. Markov's inequality gives $|q'(t)|\le D^2W/\ell$ for every $t\in[-\ell,\ell]$. Every point of $[-\ell,\ell]$ lies within $1/2$ of a sampled integer, so
\[
 W\le \frac{b-a}{2}+\frac{D^2}{2\ell}W\le\frac{b-a+W}{2},
 \qquad W\le b-a.
\]
Let $n=\lfloor D/2\rfloor$. The even part of $q$ has the form
\[
 e(t)=\frac{q(t)+q(-t)}2=h(t^2),\qquad \deg h\le n.
\]
Thus $|h(t)|\le W$ on $[0,\ell^2]$. Markov's inequality gives $|h'(t)|\le 2n^2W/\ell^2$ throughout that interval. Since $\ell\ge1$, the mean value bound on $[0,1]$ and $p(-1)\ge p(0)$ yield
\[
\begin{aligned}
 p(1)-p(0)&\le p(1)+p(-1)-2p(0)\\
 &=2(h(1)-h(0))\\
 &\le\frac{4\lfloor D/2\rfloor^2}{\ell^2}(b-a)
 \le\frac{D^2}{\ell^2}(b-a).
\end{aligned}
\]
For $D=1$, $h$ is constant and the bound is zero. The argument also covers $a=b$.
\end{proof}

The estimate is needed only on the cell containing a global minimizer, and holds for any representative of that cell. Applying it also to $-f$ gives the following simultaneous approximation of both extrema. The general conversion from a polynomial norming inequality to range-relative optimization is given in~\cite[Proposition~1]{PV}; here we use the discrete-minimum estimate directly. The reflecting geometry supplies the feasible integer samples needed here.

\begin{proposition}[Representatives approximate extrema]\label{prop:extrema}
Choose one integer representative from each integer-feasible closed cell, forming $Y$. If $P\cap\Z^d\ne\varnothing$ and $\ell\ge D^2$, then every real polynomial $f$ of degree at most $D$ satisfies
\[
\begin{aligned}
 0&\le\min_Y f-\min_{P\cap\Z^d} f\le\frac{D^2}{\ell^2}(f_{\max}-f_{\min}),\\
 0&\le\max_{P\cap\Z^d} f-\max_Y f\le\frac{D^2}{\ell^2}(f_{\max}-f_{\min}).
\end{aligned}
\]
\end{proposition}
\begin{proof}
Let $\mathbf{x}_{\min}$ minimize $f$ on $P\cap\Z^d$, and let $\mathbf{y}$ represent a cell containing it. Put $p(t)=f(\mathbf{x}_{\min}+t(\mathbf{y}-\mathbf{x}_{\min}))$. Lemma~\ref{lem:reflection} makes all integer samples with $|k|\le \ell$ feasible, so they lie between $f_{\min}$ and $f_{\max}$. Minimality gives $p(-1)\ge p(0)$. Lemma~\ref{lem:markov} now bounds $f(\mathbf{y})-f(\mathbf{x}_{\min})$. Apply the same argument to $-f$ for the maximum.
\end{proof}

\section{Algorithm and complexity}\label{sec:algorithm}

Algorithm~\ref{alg:sampling} enumerates the refined cells and compares their representatives. It accepts any degree bound $D\ge\max\{1,\deg f\}$; Theorem~\ref{thm:main} computes this bound from $f$. A constant objective requires only one integer feasibility query, and dimension zero is immediate. Representatives may be processed as a stream or saved for reuse.

\begin{algorithm}[tbp]
\centering
\begin{tikzpicture}
\node[draw=cellblue!65,rounded corners=2pt,fill=cellblue!3,inner sep=7pt] {
\begin{minipage}{.91\columnwidth}\small
\textbf{Refined-cell sampling}\par\smallskip
\textbf{Input:} bounded rational $P$, polynomial $f$, degree bound $D\ge\max\{1,\deg f\}$, accuracy $0<\eps\le1$.
\begin{enumerate}[leftmargin=1.5em,itemsep=2pt,topsep=4pt]
\item Clear row denominators. If $P=\varnothing$, return \textsf{INFEASIBLE}; otherwise remove zero constraint rows. If $f$ is constant, solve integer feasibility and return a feasible point or \textsf{INFEASIBLE}.
\item Form the shifted slacks, compute $\Phi$, set $\ell=\max\{D^2,\lceil D/\sqrt\eps\rceil\}$, and construct the thresholds~\eqref{eq:thresholds}.
\item Enumerate closed chambers in $Q$, obtaining cells $C$. Form $\widehat C$ by rounding slack bounds; if integer-feasible, obtain $\mathbf z_C\in\widehat C\cap\Z^d$.
\item If no representative exists, return \textsf{INFEASIBLE}; otherwise return one minimizing $f(\mathbf z_C)$, evaluated exactly.
\end{enumerate}
\end{minipage}};
\end{tikzpicture}
\caption{The full sampling construction depends only on $P,D,\eps$. The constant-objective shortcut returns one feasible point without constructing the reusable representative set.}
\label{alg:sampling}
\end{algorithm}

\begin{proof}[Correctness in Theorem~\ref{thm:main}]
The constant-objective shortcut is exact. Otherwise, every integer-feasible cell contributes a feasible representative, and the closed cells cover $P\cap\Z^d$. Thus $Y=\varnothing$ if and only if $P\cap\Z^d=\varnothing$. Otherwise, Proposition~\ref{prop:extrema} and $\ell=\max\{D^2,\lceil D/\sqrt\eps\rceil\}$ give
\[
 0\le\min_Y f-f_{\min}\le \frac{D^2}{\ell^2}(f_{\max}-f_{\min})\le\eps (f_{\max}-f_{\min}).
\]
The algorithm returns a minimizer on $Y$. In the full sampling construction, the cover and representatives depend only on $P,D,\eps$, so the same set works for every polynomial of degree at most $D$.
\end{proof}

It remains to account for bit operations. There are three costs: constructing the arrangement, producing integer representatives, and evaluating the objective. The distinction between the number of cells and the sizes of their descriptions is important: refinement increases the former polynomially in $1/\sqrt\eps$, whereas it increases the threshold bit lengths only logarithmically.
Lemma~\ref{lem:arrangement} bounds the chamber enumeration cost. The bit length of $\ell$ is $O(\log(D+1)+\langle\eps\rangle)$; compute the least positive integer $k$ satisfying $\eps k^2\ge D^2$ by binary search with rational comparisons, and set $\ell=\max\{D^2,k\}$. Generating all thresholds uses $O(\Phi\ell)$ arithmetic operations on integers with $O(\Phi+\log(\ell+1))$ bits and is included in the arrangement bound when $m\ge1$. A nonempty bounded polyhedron in positive dimension cannot be specified by zero inequalities.

Once a cell has been enumerated, its integer feasibility query has a description whose coefficient sizes do not grow with the refinement parameter.
Each of at most $R_d(H)$ queries has $d$ integer variables and a total encoding size
\begin{equation}\label{eq:ilp-size}
 S_C=O\!\left(m(d+1)(\tau+\Phi+1)+d+m\right).
\end{equation}
Fixed-dimensional integer linear feasibility with witness production takes $2^{O(d\log(d+1))}\poly(S_C)$ bit operations using Kannan's algorithm~\cite{Kannan}, improving the dimension dependence of Lenstra's algorithm~\cite{Lenstra}. Every integer point returned from a shifted cell belongs to the original bounded polytope. Accordingly, every returned coordinate has at most $\Phi+O(1)$ bits by~\eqref{eq:integer-bounds}; its bit length grows logarithmically with its magnitude. Notice also that~\eqref{eq:ilp-size} has no dependence on $\ell$ other than the number of cells, because the threshold inequalities have already been rounded to integer slack bounds.

Finally, the coordinate bound ensures that exact objective evaluation also has polynomial bit cost, even though feasible coordinates can be numerically large.
The integer coefficients of $f$ have bit length at most $\langle f\rangle$. If $\mathbf{x}\in P\cap\Z^d$, every monomial has absolute value at most $2^{D\Phi}$. Each value of $f(\mathbf{x})$, and the intermediates obtained by evaluating monomials and summing them, has $O(\langle f\rangle+D\Phi+\log(|\mathcal A|+1))$ bits. Repeated squaring evaluates the explicit exponent vectors in a number of arithmetic operations polynomial in $\langle f\rangle+d+\log(D+1)$. Even schoolbook integer arithmetic therefore gives
\begin{equation}\label{eq:eval}
 \poly(\langle f\rangle+d+D\Phi)
\end{equation}
bit operations per representative. The case of the zero polynomial is immediate. Returning an exact integer objective value also fits this bound. No unit-cost assumption on rational arithmetic is used.

Combining enumeration, integer feasibility, and evaluation gives the bit-operation bound
\[
\begin{aligned}
 T\le{}&O((H+1)R_d(H))\,
          \poly(\varphi,\log(\ell+1))\\
       &+R_d(H)\,2^{O(d\log(d+1))}\poly(\varphi)\\
       &+R_d(H)\,\poly(\langle f\rangle+d+D\Phi).
\end{aligned}
\]

By Lemma~\ref{lem:arrangement}, the number of cells and enumeration steps is $(\varphi\ell)^{O(d)}$, since $m,d,\tau=O(\varphi)$ and $\Phi=O(d^3(\tau+1))$. Under our unary-exponent encoding, $D=O(\varphi)$ and all auxiliary bit lengths are polynomially bounded in $\varphi$. Each enumeration test and objective evaluation therefore has polynomial bit cost, and each integer-feasibility query takes $2^{O(d\log(d+1))}\poly(\varphi)$ bit operations. Substituting $\ell=\max\{D^2,\lceil D/\sqrt\eps\rceil\}$ gives the total running time $(\varphi D/\eps)^{O(d)}$, proving Theorem~\ref{thm:main}. For binary-encoded exponents, the per-query and evaluation bounds instead retain a factor $\poly(\varphi,D)$, which is also absorbed into this total bound.

For fixed $P,D,d$, there are $O(\eps^{-d/2})$ integer-feasibility queries and $O(\eps^{-(d+1)/2})$ LP enumeration tests. These are call counts, not bit-operation bounds. Objective coefficients enter only through exact evaluation~\eqref{eq:eval}. Normalization preserves polynomial encoding size, so the bound also applies to rational inequalities.

\begin{proof}[Proof of Corollary~\ref{cor:max}]
Applying Proposition~\ref{prop:extrema} to $-f$ gives $f_{\max}-f(\mathbf{x}_{\eps})\le\eps (f_{\max}-f_{\min})$. Nonnegativity on $P\cap\Z^d$ implies $f_{\max}-f_{\min}\le f_{\max}$.
\end{proof}

\begin{proof}[Proof of Corollary~\ref{cor:mixed}]
If $d_c=0$, apply Theorem~\ref{thm:main}. Otherwise let $S=P\cap(\R^{d_c}\times\Z^{d_i})$. By~\cite[Lemma~9]{DHKW08}, an integer $\Delta>0$ of polynomial encoding size in fixed dimension can be computed such that every nonempty slice $\Delta P_{\mathbf z}$ is integral. Our modification begins with their barycentric rounding~\cite[Lemma~10]{DHKW08}: instead of rounding down all but one weight, we round up all but a largest weight. This gives the controlled feasible segment below, on which Markov's inequality replaces their coefficient-based Lipschitz estimate~\cite[Lemma~11]{DHKW08}. Set
\[
 k=\left\lceil\frac{4D^2d_c(d_c+1)}{\eps}\right\rceil,
 \qquad Q=k\Delta,\qquad t=\frac{d_c(d_c+1)}k.
\]
For any $\mathbf x\in P_{\mathbf z}$, Carath\'eodory's theorem gives $\mathbf x=\sum_{i=0}^{d_c}\lambda_i\mathbf v_i$ for slice vertices $\mathbf v_i\in\Delta^{-1}\Z^{d_c}$ and convex weights, indexed so that $\lambda_0\ge1/(d_c+1)$. Put
\[
 \mu_i=\frac{\lceil k\lambda_i\rceil}{k}\ (i\ge1),
 \qquad \mu_0=1-\sum_{i=1}^{d_c}\mu_i.
\]
For $i\ge1$, $\lambda_i\le\mu_i\le\lambda_i+1/k$, while
$\mu_0\ge\lambda_0-d_c/k\ge(1-t)\lambda_0$. Since $0<t\le1$, all weights are nonnegative and $\mu_i\ge(1-t)\lambda_i$. Therefore
\[
 \mathbf y=\sum_{i=0}^{d_c}\mu_i\mathbf v_i
 =(1-t)\mathbf x+t\mathbf w\in P_{\mathbf z}\cap Q^{-1}\Z^{d_c}
\]
for some $\mathbf w\in P_{\mathbf z}$: its convex weights are $(\mu_i-(1-t)\lambda_i)/t$.

Write $R=f_{\max}-f_{\min}$ over $S$. The polynomial
$p(u)=f((1-u)\mathbf x+u\mathbf w,\mathbf z)-f_{\min}$ has degree at most $D$ and lies in $[0,R]$ on $[0,1]$. Markov's inequality gives
\[
 |f(\mathbf y,\mathbf z)-f(\mathbf x,\mathbf z)|
 \le2D^2tR\le\eps R/2.
\]
Thus the grid preserves nonemptiness and approximates the minimum within $\eps R/2$, including when $R=0$. Apply Theorem~\ref{thm:main} on the grid with accuracy $\eps/2$. Its objective range is at most $R$, so the total error is at most $\eps R$.

Substitute $\mathbf x=\mathbf u/Q$ and use the integer polynomial $g(\mathbf u,\mathbf z)=Q^D f(\mathbf u/Q,\mathbf z)$. Its degree and support are unchanged, and coefficient bit lengths increase by $O(D\log(Q+1))$. Since
\[
 \log Q=\log\Delta+
 O\!\left(\log(D+1)+\log(d_c+1)+\log(1/\eps)\right),
\]
positive scaling preserves the guarantee; negation and nonnegativity give the maximization claim. For explicit complexity, normalize the inequalities to integers and let $\tau$ bound their coefficient bit lengths. DHKW construct $\Delta$ as the least common multiple of nonzero absolute determinants of $d_c$-row bases of the continuous-variable matrix~\cite[proof of Lemma~9]{DHKW08}. Hadamard's inequality gives
\[
 \log\Delta=O\!\left(\binom{m}{d_c} d_c(\tau+\log(d_c+1))\right).
\]
Enumerating these bases and computing their least common multiple takes $\varphi^{O(d_c+1)}$ bit operations; normalization changes encoding size only polynomially. The scaled input has size at most $\varphi^{O(d_c+1)}$ under our unary-exponent convention. Applying Theorem~\ref{thm:main} proves~\eqref{eq:mixed-runtime}. For binary exponents, the extra polynomial dependence on $D$ is likewise absorbed into the displayed bound.

The new grid multiplier avoids the coefficient Lipschitz constant and positive range-gap bound in DHKW's reduction. This sharpens the sufficient grid resolution, but does not alone establish a better overall asymptotic runtime: the integer solvers depend on $\log Q$, and the common denominator $\Delta$ remains.
\end{proof}

\section{Discussion and extensions}\label{sec:conclusion}
For any refinement $\ell\ge1$, the representatives include every vertex of the integer hull. Indeed, if such a vertex $\mathbf v$ shared a cell with a distinct integer point $\mathbf y$, reflection would make $2\mathbf v-\mathbf y$ feasible, contradicting extremality of $\mathbf v=(\mathbf y+(2\mathbf v-\mathbf y))/2$. Thus every linear objective is optimized exactly.

The reflecting cover is independent of objective coefficients, so its representatives can be reused for any polynomial of degree at most $D$. Multiplicative slack bins connect the construction to approximate Pareto methods~\cite{PY} and affine slicing~\cite{HWZ}; Markov-based norming meshes provide the continuous sampling analogue~\cite[Propositions~1--3]{PV}. Our error certificate uses feasible integer samples and the integer objective range.

CHKM also give a lattice-point counting FPTAS~\cite[Theorem~3.1]{CHKM}, combining volume estimates with recursion on lower-dimensional slices. Replacing volume estimates by Markov-controlled sampling suggests another route to polynomial optimization, but its guarantee and bit complexity require further analysis.

More generally, if $|f-h_C|\le\delta_C$ on each feasible cell and a candidate minimizes $h_C$ within $\eta_C$, choosing the best candidate gives error at most $2\delta_{C^*}+\eta_{C^*}$, where $C^*$ contains an optimum. This follows by comparing through $h_{C^*}$ and applying the two residual bounds. An FPTAS additionally requires polynomial construction, local computation, and encoding costs.

For $f(\mathbf{x})=g(B\mathbf{x})+\mathbf c^\top\mathbf{x}$ with few rows in $B$, affine optimization might handle the remaining directions. A rank-sensitive cover with sufficient feasible samples remains open; changing the oracle alone does not reduce our cell count. Nonpolyhedral extensions likewise require feasible progressions and efficient integer-feasibility queries; an outer polyhedral relaxation alone is insufficient. For continuous sets, polynomial images of compact Markov sets remain Markov under a full-generic-rank hypothesis~\cite[Theorem~1.4]{Pierzchala}, yielding further norming meshes~\cite[Proposition~2]{PV}. This supplies neither integer representatives nor uniform encoding bounds for the Markov constants.

\section*{AI assistance and formal verification}
The author used OpenAI's Codex assistant extensively in developing this paper. The starting point was the author's proposal to apply the CHKM reflecting-set decomposition, as used in the Ari--Hildebrand work, to polynomial approximation. The author contributed the problem formulation and choice of geometric approach, supplied relevant literature, and worked through the proposed arguments in an ongoing exchange with the assistant. He also guided the encoding and complexity analysis and the presentation of the results.

The assistant helped develop and revise proofs, carry out calculations, investigate related literature, and prepare the manuscript, Lean formalizations, and computational examples. Reviews from additional AI-assisted discussions were used to identify points requiring further examination and to suggest improvements. AI assistance therefore included substantive mathematical work as well as exposition. The author selected and revised the material included in the paper and is responsible for its claims and attribution. The formal checks provide additional evidence for the specific statements described below; they do not certify the entire manuscript.

The accompanying Lean development contains 256 checked results covering the sampled approximation bound, reflecting-cell geometry, mixed-integer reduction, and encoding and runtime calculations under explicit hypotheses. Classical Markov inequalities, published geometric inputs, oracle implementations, and the complete algorithmic bit-complexity analysis remain external; the full FPTAS is not formally verified. The supplementary verification guide maps the paper to the Lean statements, assumptions, and reproduction steps; the \href{https://github.com/RobertHildebrand/FPTAS-Polynomial-Optimization/tree/main/verification}{repository} provides the source and detailed verification scope.

\end{document}